\documentclass[11pt,reqno]{amsart}
\usepackage{amssymb,amsthm,amsmath,tikz}

\begin{document}

\newtheorem{thm}{Theorem}[section]
\newtheorem{cor}[thm]{Corollary}
\newtheorem{lem}[thm]{Lemma}
\newtheorem{prop}[thm]{Proposition}
\theoremstyle{definition}
\newtheorem{defn}[thm]{Definition}

\newcommand{\bm}[1]{\mbox{\boldmath $#1$}}
\newcommand{\ceil}[1]{\left\lceil #1 \right\rceil}
\newcommand{\floor}[1]{\left\lfloor #1 \right\rfloor}
\newcommand{\C}{\ensuremath C}
\newcommand{\e}{\ensuremath e}
\renewcommand{\l}{\ensuremath\lambda}
\newcommand{\peq}{\preceq}

\title{Optimizing Linear Extensions}

\author{Bridget Eileen Tenner}
\email{bridget@math.depaul.edu}
\address{Department of Mathematical Sciences, DePaul University, Chicago, Illinois 60614}


\begin{abstract}
The minimum number of elements needed for a poset to have exactly $n$ linear extensions is at most $2\sqrt{n}$.  In a special case, the bound can be improved to $\sqrt{n}$.
\end{abstract}

\maketitle

\section{Introduction and definitions}

A partially ordered set, or \emph{poset}, $P = (X, \peq)$ consists of a set $X$ together with a partial ordering $\peq$ on $X$.  For background on these structures, the reader is encouraged to review \cite{ec1} and \cite{trotter}.

One statistic that can hint at how much information is missing in a partial ordering is based on the following definition.

\begin{defn}
A \emph{linear extension} of a poset $P = (X, \peq)$ is a total ordering of the elements of $X$ that is compatible with $\peq$.  The number of linear extensions of $P$ is denoted $\e(P)$.
\end{defn}

As suggested in \cite{ec1}, the number of linear extensions of a poset gives an indication of the intricacy of the original partial ordering.  Thus understanding the function $\e$ can provide some insight into the complexity of the structure of partial orderings.

Another poset statistic, the number of order ideals in a poset, is considered in \cite{rt}, and a bound is given for the minimal number of elements needed to have a particular number of order ideals.  Here, the analogous question is answered for the function $\e$.

\begin{defn}
The \emph{size} of a poset $P = (X, \peq)$, denoted $|P|$, is the cardinality of $|X|$.
\end{defn}

\begin{defn}
For any integer $n \ge 1$, set $\l(n) = \min\{|P| : e(P) = n\}$.
\end{defn}

The main result of this work, Theorem~\ref{thm:main}, is the bound
$$\l(n) \le 2\sqrt{n}.$$
In a certain case, as discussed in Section~\ref{section:special cases}, this bound can be improved further to $\sqrt{n}$.  As displayed in Table~\ref{table:examples}, there are values of $n$ for which $\l(n)$ equals $2\sqrt{n}$.

In the next section, the values of $\l(n)$ for small $n$ are given, together with examples of the posets that obtain them.  Furthermore, the poset operations that give the primary tools for proving Theorem~\ref{thm:main} are stated.  Section~\ref{section:main} consists of the main result, and a special case is treated in the last section.

\section{Examples and arithmetic of poset operations}

Before describing how basic poset operations affect the function $\l$, it is instructive to calculate $\l(n)$ for some small values of $n$, and to view the posets that give these values.  These examples appear in Table~\ref{table:examples}, and as sequence A160371 in \cite{oeis}.

\begin{table}[htbp]
\begin{center}
$\begin{array}{c|c|c|c|c|c|c|c|c|c|c|c|c}
\rule[-2mm]{0mm}{6mm}n & 1 & 2 & 3 & 4 & 5 & 6 & 7 & 8 & 9 & 10 & 11 & 12\\
\hline
\rule[-2mm]{0mm}{6mm}\l(n) & 0 & 2 & 3 & 4 & 4 & 3 & 5 & 4 & 5 & 5 & 5 & 4\\
\hline
\rule[-2mm]{0mm}{16mm}
\parbox[b]{.5in}{\begin{center}poset\\example\end{center}} &
\parbox[b]{.2in}{\begin{center} $\emptyset$ \end{center}} &
\begin{tikzpicture}[scale=.5] 
\foreach \x in {0,1} {\fill[black] (\x,0) circle (3pt);}
\end{tikzpicture} &
\begin{tikzpicture}[scale=.5] 
\draw (1,0) -- (1,1);
\foreach \x in {0,1} {\fill[black] (\x,0) circle (3pt);}
\fill[black] (1,1) circle (3pt);
\end{tikzpicture} &
\begin{tikzpicture}[scale=.5] 
\foreach \x in {0,1} {\fill[black] (\x,0) circle (3pt); \fill[black] (\x,1) circle (3pt); \draw (\x,0) -- (\x,1);}
\draw (0,0) -- (1,1);
\draw (0,1) -- (1,0);
\end{tikzpicture} &
\begin{tikzpicture}[scale=.5] 
\foreach \x in {0,1} {\fill[black] (\x,0) circle (3pt); \fill[black] (\x,1) circle (3pt); \draw (\x,0) -- (\x,1);}
\draw (0,1) -- (1,0);
\end{tikzpicture} &
\begin{tikzpicture}[scale=.5] 
\foreach \x in {0,.5,1} {\fill[black] (\x,0) circle (3pt);}
\end{tikzpicture} &
\begin{tikzpicture}[scale=.5] 
\foreach \x in {0,1} {\fill[black] (\x,0) circle (3pt); \fill[black] (\x,1) circle (3pt); \draw (\x,0) -- (\x,1);}
\draw (0,1) -- (1,0);
\draw (1,0) -- (1,-1); \fill[black] (1,-1) circle (3pt);
\end{tikzpicture} &
\begin{tikzpicture}[scale=.5] 
\foreach \x in {0,1,2} {\fill[black] (\x,0) circle (3pt);}
\fill[black] (.5,1) circle (3pt);
\draw (0,0) -- (.5,1); \draw (1,0) -- (.5,1);
\end{tikzpicture} &
\begin{tikzpicture}[scale=.5] 
\foreach \x in {0,1} {\fill[black] (\x,0) circle (3pt); \fill[black] (\x,1) circle (3pt); \draw (\x,0) -- (\x,1);}
\draw (0,1) -- (1,0);
\draw (0,0) -- (0,-1); \fill[black] (0,-1) circle (3pt);
\end{tikzpicture} &
\begin{tikzpicture}[scale=.5] 
\foreach \x in {0,1,2} {\fill[black] (\x,0) circle (3pt);}
\foreach \y in {1,2} {\fill[black] (.5,\y) circle (3pt);}
\draw (0,0) -- (.5,1); \draw (1,0) -- (.5,1);
\draw (.5,1) -- (.5,2);
\end{tikzpicture} &
\begin{tikzpicture}[scale=.5] 
\foreach \x in {0,1,2} {\fill[black] (\x,0) circle (3pt);}
\foreach \y in {-1,1} {\fill[black] (1,\y) circle (3pt);}
\draw (1,-1) -- (1,1);
\draw(0,0) -- (1,1); \draw (1,-1) -- (2,0);
\end{tikzpicture} &
\begin{tikzpicture}[scale=.5] 
\draw (1,0) -- (1,1);
\foreach \x in {0,1,2} {\fill[black] (\x,0) circle (3pt);}
\fill[black] (1,1) circle (3pt);
\end{tikzpicture}\\
\hline
\rule[-2mm]{0mm}{6mm}\floor{2\sqrt{n}} & 2 & 2 & 3 & 4 & 4 & 4 & 5 & 5 & 6 & 6 & 6 & 6
\end{array}$
\end{center}
\vspace{.1in}
\caption{The values of $\l(n)$ for $1 \le n \le 12$, together with demonstrative posets, and the upper bound of Theorem~\ref{thm:main}.}\label{table:examples}
\end{table}

Two elementary operations on posets are the \emph{direct sum} and the \emph{ordinal sum}.  Note that a poset which can be constructed entirely by these two operations is called \emph{series-parallel}.

\begin{defn}
Let $P$ and $Q$ be posets on the sets $X_P$ and $X_Q$, respectively, with order relations $\peq_P$ and $\peq_Q$, respectively.  The direct sum $P + Q$ is the poset defined on $X_P \cup X_Q$, with order relations $\peq_P \cup \peq_Q$.  The ordinal sum $P \oplus Q$ is the poset defined on $X_P \cup X_Q$, with order relations $\peq_P \cup \peq_Q \cup \{x_P \peq x_Q : x_P \in X_P, x_Q \in X_Q\}$.
\end{defn}

The next lemma follows immediately from the definitions.

\begin{lem}\label{lem:basic operations}
For posets $P$ and $Q$,
$$\e(P+Q) = \binom{|P|+|Q|}{|P|} \e(P) \e(Q)$$
and
$$\e(P \oplus Q) = \e(P) \e(Q).$$
\end{lem}

\begin{defn}
For any $\ell \ge 0$, let the poset $\C_\ell$ be the chain of $\ell$ elements, where $\C_0 = \emptyset$.
\end{defn}

Certainly the poset $\C_\ell$ is already a total ordering, so $\l(\C_\ell) = 1$ for all $\ell$.  Moreover, it follows from the identities of Lemma~\ref{lem:basic operations} that
$$\e(P + \C_\ell) = \binom{|P| + \ell}{|P|} \e(P)$$
and
\begin{equation}
\e(P \oplus \C_\ell) = \e(P) \label{eqn:chain ordinal sum}
\end{equation}
for all $\ell \ge 0$.  Equation~\eqref{eqn:chain ordinal sum} implies that a poset with $n$ linear extensions can have arbitrarily large size.  Perhaps unexpectedly, equation~\eqref{eqn:chain ordinal sum} will be very helpful in bounding $\l(n)$.  The key is to employ it as in the following result.

\begin{prop}\label{prop:factoring}
For all $\ell \ge 0$, $\e\left((P \oplus \C_\ell) + \C_1\right) = (|P| + \ell+1)\e(P).$
\end{prop}

\begin{figure}[htbp]
\begin{center}
\begin{tikzpicture}[scale=.5]
\tikzstyle{block} = [rectangle, draw, text width=.75cm, text centered, rounded corners, minimum height=.75cm]
\draw (0,.75) coordinate -- (0,3) coordinate (b);
\draw (0,1.5) coordinate (a) -- (-.6,.75);
\draw (a) -- (.6, .75);
\node[block] at (0,0) (P) {$P$};
\fill[black] (a) circle (2pt);
\fill[black] (b) circle (2pt);
\fill[black] (2,0) circle (2pt);
\draw[<->] (-.5,1.5) -- (-.5,3);
\node[left] at (-.5,2.25) {$\ell$};
\end{tikzpicture}
\end{center}
\caption{The poset $(P \oplus \C_\ell) + \C_1$ described in Proposition~\ref{prop:factoring}.}\label{fig:factoring}
\end{figure}
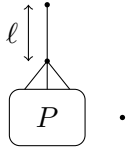

Proposition~\ref{prop:factoring} gives the following initial result for all $n$.

\begin{cor}\label{cor:factoring}
If $n = ab$ for $a,b \in \mathbb{Z}^+$ with $a < b$, then $\l(n) \le b$.
\end{cor}

\begin{proof}
First note that
\begin{equation}\label{eqn:initial bound}
\l(n) \le n
\end{equation}
for all $n \in \mathbb{Z}^+$, by considering the $n$-element poset $\C_{n-1} + \C_1$, which has $n$ linear extensions.

Let $P$ be a poset of size $\l(a)$, with $\e(P) = a$.  Since $a < b$, equation~\eqref{eqn:initial bound} implies $\l(a) < b$, and so $b - 1 - |P| \ge 0$.  Set $Q = (P \oplus \C_{b - 1 - |P|}) + \C_1$.  Then $|Q| = |P| + b - 1 - |P| + 1 = b$, and $\e(Q) = (|P| + b - 1 - |P| + 1) \e(P) = ab = n$.
\end{proof}


\section{Bounds}\label{section:main}

The proof of the main result, Theorem~\ref{thm:main}, begins with an analysis of the following $m$-element poset $Q_{i,j,m}$, where $1 \le i < j \le m-2$.  Note that $Q_{i,j,m}$ is not series-parallel.

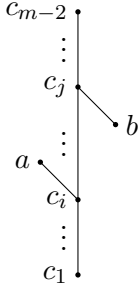
\begin{center}
\begin{tikzpicture}[scale=.5]
\draw (0,2) coordinate (c_1) -- (0,9) coordinate (c_m);
\coordinate (c_i) at (0,4);
\coordinate (c_j) at (0,7);
\fill[black] (0,2) circle (2pt) node[left] {$c_1$};
\fill[black] (0,4) circle (2pt) node[left] {$c_i$};
\fill[black] (0,7) circle (2pt) node[left] {$c_j$};
\fill[black] (0,9) circle (2pt) node[left] {$c_{m-2}$};
\draw (-1,5) coordinate (a) -- (0,4);
\draw (1,6) coordinate (b) -- (0,7);
\fill[black] (-1,5) circle (2pt) node[left] {$a$};
\fill[black] (1,6) circle (2pt) node[right] {$b$};
\foreach \y in {3.2,5.7,8.2} {\draw (0,\y) node[left] {$\vdots$};}
\end{tikzpicture}
\end{center}

\noindent In any linear extension of $Q_{i,j,m}$, the elements $\{c_1,c_2, \ldots, c_{m-2}\}$ may appear in exactly one order.  The element $a$ can appear anywhere after $c_i$, while the element $b$ can appear anywhere before $c_j$.  The elements $a$ and $b$ are incomparable in $Q_{i,j,m}$, so they can appear in either order if they both appear between $c_k$ and $c_{k+1}$ in a linear extension.  Thus
$$\e(Q_{i,j,m}) = (m-1-i)j + (j-i) = (m-i)j-i,$$
and so
$$\l\left((m-i)j-i\right) \le m.$$

\begin{prop}\label{prop:linear bound}
For all integers $n \ge 1$ and $d \ge 1$, 
$$\l(n) \le \floor{n/d}+d.$$
\end{prop}

\begin{proof}
This is proved by induction on $d$, where the case $d=1$ follows from equation~\eqref{eqn:initial bound}.

Now suppose that $d \ge 2$ and that the result holds for all $d' \in [1, d)$.  The integer $n$ can be written as $n = qd - r$, where $r \in [0, d-1]$.  If $r \ge 1$ and $q+r -2 \ge d$, then $Q_{r, d, q+r}$ is a poset having $n$ linear extensions and size
$$q+r \le \floor{n/d}+1 + (d-1) = \floor{n/d} + d.$$
Thus it remains to consider when $r = 0$ or $q+r-2 < d$.

If $r = 0$, then $n = qd$ and Lemma~\ref{lem:basic operations} implies that
$$\l(n) \le \l(q) + \l(d) \le q + d = \floor{n/d} + d.$$
This leaves the case when $r \in [1, d-1]$ and $q + r - 1 \le d$.  The few cases that remain when $d \in \{2,3\}$ are easy to check (in fact, they concern only $n \le 12$, and so appear in Table~\ref{table:examples}).  For the conclusion of the argument, suppose $d \ge 4$.

Rewrite $n$ as $n = q'(d-1) + r'$ where $r' \in [0, d-2]$.  Because $n = q(d-1) + q + r$, the restrictions on $q$, $r$, and $d$ imply that there is at most one extra factor of $d-1$ in $q+r$.  That is, $q' \in \{q,q+1\}$.  From the induction hypothesis for $d' = d-1$, it follows that $\l(n) \le q' + d - 1 \le q+d$, which completes the proof.
\end{proof}

Although the bound in Proposition~\ref{prop:linear bound} is linear, the fact that it holds for all integers $d \ge 1$ indicates that it can be improved further.

\begin{thm}\label{thm:main}
For all $n \ge 1$, $\l(n) \le 2\sqrt{n}$.
\end{thm}

\begin{proof}
Apply Proposition~\ref{prop:linear bound} with $d = \lceil \sqrt{n} \rceil$ and $\varepsilon = \ceil{\sqrt{n}} - \sqrt{n}$, where $\varepsilon \in [0,1)$:
\begin{equation}\label{eqn:floor/ceiling}
\l(n) \le \floor{\frac{n}{\ceil{\sqrt{n}}}} + \ceil{\sqrt{n}} = \floor{\sqrt{n} - \varepsilon + \frac{\varepsilon^2}{\sqrt{n} + \varepsilon}} + \sqrt{n} + \varepsilon.
\end{equation}

If $\varepsilon = 0$, then $d = \sqrt{n}$, and the theorem holds.  If $\varepsilon \in (0,.5]$, then $\varepsilon - 1 \le -\varepsilon$, and
$$\floor{\sqrt{n} - \varepsilon + \frac{\varepsilon^2}{\sqrt{n} + \varepsilon}} \le \floor{\sqrt{n}} = \sqrt{n} + \varepsilon - 1 \le \sqrt{n} - \varepsilon.$$
On the other hand, if $\varepsilon \in (.5,1)$, then $\varepsilon - 2 < -\varepsilon$, and
$$\floor{\sqrt{n} - \varepsilon + \frac{\varepsilon^2}{\sqrt{n} + \varepsilon}} < \floor{\sqrt{n} - \varepsilon + \frac{1}{2}} \le \floor{\sqrt{n}}.$$
In other words, if $\varepsilon \in (.5, 1)$, then
$$\floor{\sqrt{n} - \varepsilon + \frac{\varepsilon^2}{\sqrt{n} + \varepsilon}} \le \floor{\sqrt{n}} - 1 = \sqrt{n} + \varepsilon - 2 < \sqrt{n} - \varepsilon.$$

Therefore, for any $\varepsilon \in [0,1)$, it follows from inequality~\eqref{eqn:floor/ceiling} that $\l(n) \le 2\sqrt{n}$.
\end{proof}

\section{A special case}\label{section:special cases}

As suggested in Corollary~\ref{cor:factoring}, the number $\l(n)$ is influenced by the factorization of $n$.  In particular, primality of $n$ can be a challenge for the function $\l$.  On the other hand, if $n$ factors in a particular way, then the bound on $\l(n)$ can be further tightened along the lines of Corollary~\ref{cor:factoring}.

\begin{cor}
If $n = ab$ for $a,b \in \mathbb{Z}^+$ with $2\sqrt{b} < a \le b$, then $\l(n) \le \sqrt{n}$.
\end{cor}

\begin{proof}
Suppose that $n = ab$, where $1 \le a \le b < (a/2)^2$.  Construct a poset $P$ with $e(P) = b$ and $|P| = \l(b) \le 2\sqrt{b} < a$.  Let $Q = (P \oplus \C_{a-1-|P|}) + \C_1$.  Note that $e(Q) = ab = n$ and $|Q| = a$.  Since $n = ab$ and $a \le b$, this implies that $|Q| \le \sqrt{n}$, and so $\l(n) \le \sqrt{n}$.
\end{proof}


\begin{thebibliography}{9}

\bibitem{rt} K.~Ragnarsson and B.~E.~Tenner, Obtainable sizes of finite topologies, to appear in \textit{J.~Combin.~Theory, Series A}.

\bibitem{oeis}  N.~J.~A.~Sloane, The on-line encyclopedia of integer sequences, published electronically at \hfill \phantom{*} {\tt http:/$\!\!$/www.research.att.com/\~{}njas/sequences/}.

\bibitem{ec1} R.~P.~Stanley, \textit{Enumerative Combinatorics, vol.~1}, Cambridge Studies in Advanced Mathematics, no.~49, Cambridge University Press, Cambridge, 1997.

\bibitem{trotter} W.~T.~Trotter, \textit{Combinatorics and Partially Ordered Sets: Dimension Theory}, Johns Hopkins University Press, Baltimore, 1992.

\end{thebibliography}
\end{document}